\documentclass[reqno]{amsart}
\usepackage{multirow}
\usepackage{graphics}
\usepackage{epigraph}

\newtheorem{theorem}{Theorem}[section]
\newtheorem{lema}[theorem]{Lemma}
\newtheorem*{thm}{Theorem}

\newtheorem{cor}[theorem]{Corollary}

\newtheorem{rmk}[theorem]{Remark}

\theoremstyle{remark}

\newcommand{\QQ}{\mathbb{Q}}

\newcommand{\ZZ}{\mathbb{Z}}
\newcommand{\CC}{\mathbb{C}}

\newcommand{\EFF}{\mathrm{Eff}}
\newcommand{\NEF}{\mathrm{Nef}}
\newcommand{\PP}{\mathbb{P}}

\newcommand{\OO}{\mathcal O}

\numberwithin{equation}{section}

\begin{document}

\title{Extremal higher codimension cycles of the space of complete conics}

\author{C\'esar Lozano Huerta}
\address{Instituto de Matem\'aticas, Universidad Nacional Aut\'onoma de M\'exico. Oaxaca de Ju\'arez, Oax. M\'exico.}

\email{lozano@im.unam.mx}
\thanks{The author is currently a CONACYT-Research Fellow at UNAM, ID: $71842$.}




\begin{abstract} Let $M$ denote the space of complete conics. We compute the cone of effective and numerically effective $k$-cycles of $M$, $\mathrm{Eff}_k(M)$ and $\mathrm{Nef}_k(M)$, respectively. In addition, we compute the Bia\l{}ynicki-Birula cell-decomposition of $M$ with respect to a $\mathbb{C}^*$-action and compare the cone generated by the closure of these cells to the cones $\mathrm{Eff}_k(M)$ and $\mathrm{Nef}_k(M)$.
\end{abstract}

\maketitle

\section*{Introduction}

\medskip\noindent
The study of how a subvariety sits and moves inside an ambient algebraic variety has driven research in algebraic geometry for a long time. This study finds a robust and well-developed theory when dealing with subvarieties of codimension $1$, which from the theoretical viewpoint is well-understood. For example, studying points on an algebraic curve leads to theorems such as the Riemann-Roch Theorem and the Abel-Jacobi Theorem. A sensible next step in understanding the extrinsic geometry of subvarieties is to study subvarieties of codimension $2$ and higher. This situation is not well-understood and is the situation we will study. We work over the field of complex numbers throughout.

\medskip\noindent
The set of all subvarieties contained in an algebraic variety is unmanageable. Thus, one seeks to impose an equivalence relation on them, or rather on the formal linear combinations of them, so we study the set of equivalence classes which may be more manageable. These formal sums are called cycles, and Severi \cite{SEV} in a collection of papers set up and investigated the concept of rational equivalence on them. Although this is a generalization inspired by the rational equivalence of divisors on curves, it exhibits distinct behavior. For example, Mumford \cite{MUMII} showed that the group of $0$-cycles, up to rational equivalence, on a surface does not necessarily have the properties of a finite-dimensional space. We will study the space of complete conics $M$ and the spaces of its rational-equivalent cycles $A_k(M)\otimes \QQ$; these are finite-dimensional vector spaces.

\medskip\noindent
The space of complete conics is a compactification of the family of smooth conics on the plane $\PP^2$. It was introduced by Chasles \cite{CHAS} in 1864 in order to solve questions in enumerative geometry and is defined as follows: the space of complete conics is the closure of the set of pairs $M=\overline{\{(C,C^*)\}}\subset \PP^5\times \PP^{5*}$, where $C$ is a smooth conic and $C^*$ its dual. This space is a smooth variety of dimension $5$, and has been used to study enumerative questions such as the number of conics tangent to five fixed conics in general position; which is $3264$ and has a long history in algebraic geometry \cite{CONICS}. The purpose of this note is to show that the enumerative geometry of conics that $M$ allows one to study dictates the structure of certain cones contained in the space of cycles $A_k(M)$: the cones of effective cycles, $\EFF_k(M)$.

\medskip\noindent
We are primarily interested in understanding subvarieties of $M$ of dimension $k$, so we will focus on the cone generated by non-negative linear combinations with rational coefficients of subvarieties of dimension $k$ in the space of $k$-cycles $A_k(M)\otimes \QQ$. This is the cone of effective $k$-cycles, $\EFF_k(M)$. In general, these cones are very difficult to compute and only few examples are known; many of them very recently \cite{DEL, CC}. We will explicitly compute these cones for the space $M$.

\medskip\noindent
These cones are well-known in dimension $1$ and $4$. Indeed, from the definition of $M$, there are $4$ divisors we can investigate: $H_i=\pi^*_i(h)$ and the strict transform $E_i=\pi^{-1}_i(\Delta)$, where $\pi_i$ ($i=1,2$) denotes the projection map to the first or second factor. Here, $h$ denotes the hyperplane section, and $\Delta$ denotes the locus of singular conics $i.e.$, a discriminant variety. The cone of effective divisors is generated by the classes $\EFF_4(M)=\langle E_1,E_2\rangle$, and the cone of numerically effective divisors, the nef cone, is generated by the classes $\overline{\EFF}_1(M)^{\vee}= \NEF^1(M)=\langle H_1,H_2\rangle$. The divisor $E_2$ parametrizes reducible conics, and $E_1$ double lines with two marked-points. A generic point on $H_1$ parameterizes a smooth conic which passes through a fixed point, and a generic point on $H_2$ parametrizes a smooth conic tangent to a fixed line. 

\medskip\noindent
The main results of this note, Theorem \ref{THM1} and Theorem \ref{THM22}, extend the description of the previous paragraph to dimensions $2$ and $3$, $i.e.$, we describe the cones $\EFF_k(M)$ for $k=2,3$. In addition, one can read off the geometric interpretation of the generators of these cones in terms of families of conics. To simplify the language, $H_iE_i$ denotes the intersection of the subvarieties $H_i$ and $E_i$.

\medskip
\begin{thm}\label{THM2} The vector spaces $A_k(M)\otimes \QQ$ for $k=2,3$, both have dimension $3$. The respective cones of effective cycles are generated by the following classes \begin{equation}
\begin{aligned}
\mathrm{Eff}_{2}(M)&=\langle H_1^2E_1,\ \! H_2^2E_2,\  \!H_1E_1E_2,\ \!H_2E_1E_2\rangle,
\nonumber\\
\mathrm{Eff}_3(M)&=\langle H_1E_1, H_2E_2 , E_1E_2 \rangle. 
\end{aligned}
\end{equation}
\end{thm}

\begin{center}
\begin{figure}[htb]
\resizebox{1\textwidth}{!}{\includegraphics{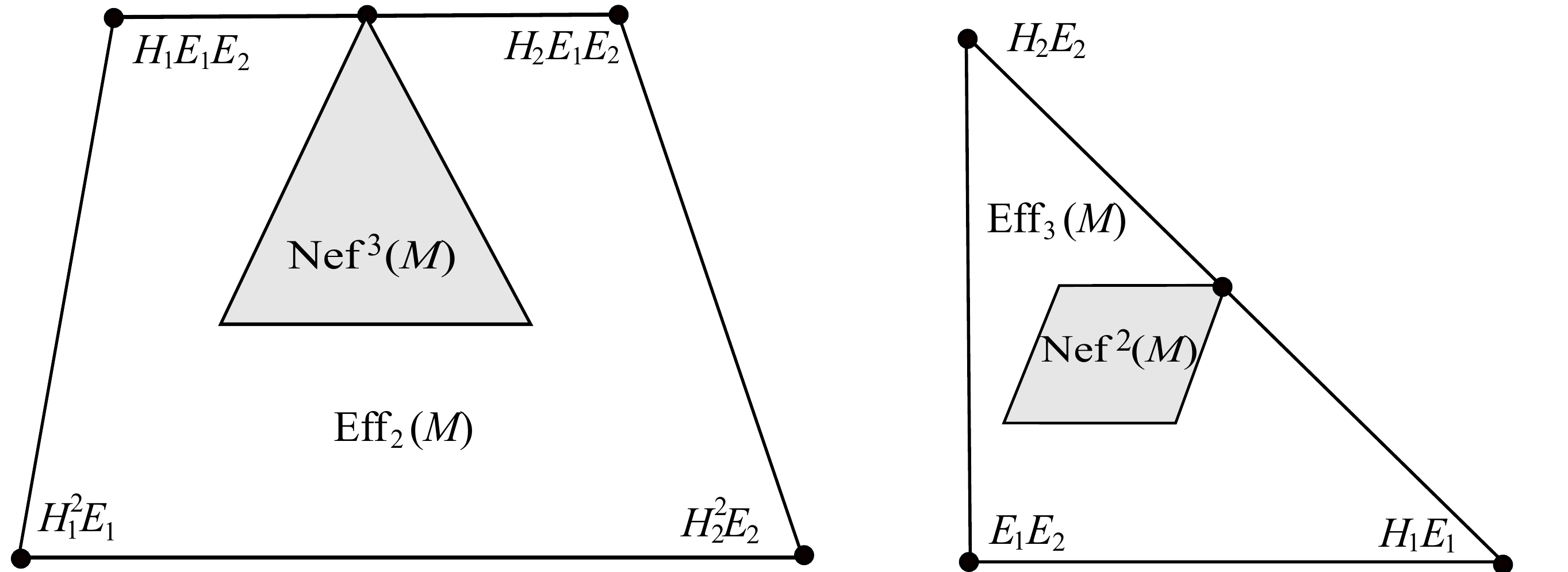}}
\caption{Cross-section of the cones $\EFF_k(M)$, and their duals $\NEF^k(M)$.}
\end{figure}
\end{center}

\medskip\noindent 
An important part of the proof is to reduce the computation of the effective cone $\EFF_k(M)$ to the computation of the effective cone of two subvarieties: $\EFF_k(E_1)$ and $\EFF_k(E_1E_2)$. We are able to do this because there is an action of $SL_3\CC$ on $M$, and $E_1$ as well as $E_1E_2$ are the closure of $SL_3$-orbits. Our technique may be applied in the more general context in which the orbits of a group action behave well; for example ``wonderful compactifications'' \cite{DECON}. Since the action of $SL_3\CC$ is ultimately the key ingredient in order to compute the effective cones, one may ask whether the action of a smaller group would yield the same results. For example:

\medskip\noindent 
The space $M$ inherits from $\PP^2$ an action of the torus $\CC^*$ which has a finite number of fixed points. Thus, it induces a Bia\l{}ynicki-Birula decomposition of $M$ into affine cells \cite{BB}. We call these cells the BB-cells of $M$. By considering the closure of the BB-cells of dimension $k$, we get effective $k$-cycles on $M$ and it is natural to ask, do they generate the cone of effective cycles of dimension $k$, $\EFF_k(M)$? 
In general, we answer in the negative by computing the number of BB-cells of dimension $2$ and the cone $\EFF_2(M)$. However, the BB-cells of $M$ of dimension $4$ do generate the cone of effective divisors $\EFF_4(M)$. In general, the precise relation between the cone generated by the BB-cells of a $\CC^*$-variety $X$ and the cones of effective cycles $\EFF_k(X)$ remains mysterious. 

\medskip\noindent  

\medskip
\section{Preliminaries} 

\medskip\noindent 
In this section, we recall definitions and properties of $M$. We
refer the reader to \cite{FULTON} for a more general treatment. A \textit{cycle} on $M$ of dimension $k$ is a formal sum of $k$-dimensional subvarieties of $M$ with rational coefficients. Such a cycle is called \textit{effective} if all the coefficients of this formal sum are nonnegative. Two cycles $X$ and $Y$ of dimension $k$ are numerically equivalent if $X\cdot\alpha=Y\cdot\alpha$ for all irreducible subvarieties of codimension $k$. Here, $\cdot$ denotes the intersection product.
Let $N_k(M)$ (respectively, $N^k(M)$) denote the finite dimensional $\QQ$-vector space of cycles of dimension $k$ (respectively, codimension $k$) up to numerical equivalence. The cone in $N_k(M)$ generated by the classes of effective cycles of dimension $k$ is called \textit{effective cone} and is denoted by $\EFF_k(M)$. 

\medskip\noindent 
A cycle $Z$ of dimension $k$ is numerically effective, or simply \textit{nef}, if $Z\cdot \alpha\ge 0$ for all irreducible subvarieties $\alpha\subset M$ of codimension $k$. One defines the \textit{nef cone} of cycle classes of dimension $k$, $\NEF_k(M)$, as the cone generated by the classes in $N_k(M)$ of nef cycles. Similarly, one defines $\NEF^k(M)$, the cone of nef cycle classes of codimension $k$. Note that in general the dual of the cone  $\NEF^k$ is the closure of the effective cone $\overline{\EFF}_k$. These convex cones are effectively described  by extremal rays. Recall that a ray $R$ is said to be an \textit{extremal ray} if for any $D\in R$ with $D=D_1+D_2$, where $D_1,D_2$ are in the cone, then $D_1, D_2\in R$. If an extremal ray is spanned by the class of an effective cycle $Z$, we say that $Z$ is an \textit{extremal effective cycle}. In general, $\EFF_k$ is not equal to its closure. In our case, we have that $\overline{\EFF}_k(M)=\EFF_k(M)$ as a consequence of the results of the present paper.

\medskip\noindent
The space of complete conics $M$ admits the following two alternative descriptions. First, $M\cong \overline{M}_{0,0}(\PP^{2},2)$ is isomorphic to a Kontsevich moduli space of stable maps. This description of $M$ often facilitates to perform intersection theory, but we will make little use of this isomorphism. On the other hand, it is important for us that $M$ is isomorphic to the blowup of $\PP^5$ along the Veronese surface $S\subset \PP^5$, \cite{JOE}. The exceptional divisor of this blowup coincides with $E_1$. In fact, from the definition of $M$ as the closure of the set of pairs $\{(C,C^*)\}$, the projection into the first factor is the blowup map. Furthermore, this description of $M$ allows one to compute $A^k(M)$ for $k=1,2,3$, and see that rational equivalence and numerical equivalence of cycles coincide on $M$. In fact, rational and numerical equivalence coincide for a large class of varieties called spherical varieties \cite{PERRIN}; $M$ is a variety of this type. Therefore, we may use freely $N_k(M)$ or $A_k(M)\otimes \QQ$.  

\medskip\noindent
The space $M$ inherits an action of $SL_3$ from $\PP^2$. This action induces the following stratification into $SL_3$-orbits \cite{DECON}  $$M=U\cup (E_1^{\circ}\cup E_2^{\circ})\cup E_1E_2,$$ where $U$ is an open dense subset, and $E_1^{\circ}$ and  $E_2^{\circ}$ are $SL_3$-orbits of codimension 1. The unique closed orbit is $E_1E_2$.

\medskip\noindent
The characteristic numbers of $M$ such as $H_2^5=1$ (the number of conics tangent to five fixed lines in general position) have a long history in algebraic geometry. We make use of such numbers in order to understand the geometry of the cones $\EFF_k(M)$ and $\NEF_k(M)$. We refer the reader to \cite{GH} for a proof of the following

\medskip\noindent
\begin{lema}\label{CHAR}
Let $M$ be the space of complete conics. The following intersection numbers hold, $$H_1^5=H_2^5=1,\quad H_1^4H_2=H_1H_2^4=2,\quad H_1^2H_2^3=H_1^3H_2^2=4. $$
\end{lema}

\medskip\noindent
Using this lemma, it is straightforward to obtain that $H_2=2H_1-E_1$ and $H_1=2H_2-E_2$. There is a $\ZZ/2$-symmetry in the previous lemma induced by the Gauss map $C\mapsto C^*$, where $C^*$ denotes the dual conic of $C$. We make use of this symmetry in order to simplify our computations.

\section{Bia\l{}ynicki-Birula Decomposition of $M$}\label{II}

\medskip\noindent
In this section, we compute the dimension of the affine cells of the Bia\l{}ynicki-Birula decomposition of $M$ following \cite{BB}. This decomposition is induced by a $\CC^*$-action on $M$ inherited from the plane $\PP^2$. Let us state, in our context, the theorem we will use.

\begin{thm}\cite{BB}
Let $X$ be a smooth projective variety. Suppose there is an action of $\CC^*$ on $X$. If this action has a finite number of isolated fixed points $\{p_1,\ldots
,p_n\}$, then
\begin{itemize}
\item There exists a decomposition $X=\bigcup_i
  C_{p_i}$, where each $C_{p_i}$ is an affine variety containing a unique $p_i$, called center.
\item The tangent space at any $p_i$, $T_{p_i}X$, becomes a $\CC^*$-module. Hence, there is a decomposition
$$T_{p_i}X=T^{+}\oplus T^{-1},$$ where $T^{-1}$ denotes the subspace upon which the $\CC^*$-action has negative weights, and $C_{p_i}\cong
T^{-1}$ as affine varieties.
\end{itemize}
\end{thm}

\medskip\noindent
Since $M\cong \mathrm{Bl}_S( \PP^5)$ is isomorphic to the blowup of $\PP^5$ along the Veronese surface $S$, we will compute the dimension of the BB-cells of $M$ by first computing the BB-cells of $\PP^5$.
Let us denote by $[a:\ldots :f]$ the homogeneous coordinates for the linear system $|\OO_{\PP^2}(2)|\cong \PP^5$. There is an induced action of $\CC^*\! \times \CC^*$ on $|{O}_{\PP^2}(2)|$ given by $[at_0^2:bt_0t_1:ct_1^2:dt_0:et_1:f]$. Let $T$ be the image of $\CC^*$ in $\CC^*\!\times \CC^*$ under the embedding $t\mapsto (t,t^k)$, $k\gg 0$. We call the action of this embedded torus $T$-action. This action of $T$ on $\PP^5$ has six isolated fixed points, $p_i=[0:\ldots :1:\ldots :0]\in \PP^5$, which implies that it induces a Bia\l{}ynicki-Birula decomposition of $\PP^5$.
According to the Theorem above, the dimension of each BB-cell is the number of negative weights of the $T$-action on each tangent space $T_{p_i}\PP^5$. The following table shows the dimension of the BB-cells, and their respective centers. It also shows $\CC^* \! \times \CC^*$ acting on each tangent space $T_{p_i}\PP^5$ by a diagonal matrix, which we denote by $\mbox{diag}$.

\begin{equation}\label{TABLE1}
\begin{tabular}{ccccc}\hline
center & cell-dimension & $\CC^*\times \CC^*$-action on $T_{p_i}\PP^5$ \\
\hline
 \\
$p_0$ & 2 & $\mbox{diag}(t_0^{-1}t_1,t_0^{-2}t_1^2,t_0^{-1},t_1t_0^{-2},t_0^{-2})$ \\[2mm]
$p_1$ & 4 & $ \mbox{diag}(t_0t_1^{-1},t_0^{-1}t_1,t_1^{-1},t_0^{-1},t_0^{-1}t_1^{-1})$ \\[2mm]
$p_2$ & 5 & $\mbox{diag}(t_0^2t_1^{-2},t_0t_1^{-1},t_0t_1^{-2},t_1^{-1},t_1^{-2})$ \\[2mm]
$p_3$ & 1 & $\mbox{diag}(t_0,t_1,t_1^{2}t_0^{-1},t_0^{-1}t_1,t_0^{-1})$  \\[2mm]
$p_4$ & 3 & $\mbox{diag}(t_0^{2}t_1^{-1},t_0,t_1,t_0t_1^{-1},t_1^{-1})$ \\[2mm]
$p_5$ & 0 & $\mbox{diag}(t_0^2,t_0t_1,t_1^{2},t_0,t_1)$ \\ \hline
\end{tabular}
\end{equation}

\medskip\noindent
In order to describe the BB-decomposition of the space of complete conics $M\cong \mathrm{Bl}_S( \PP^5)$, let us first determine the fixed points of the action of $T$ on $M$. Since the surface $S$ is invariant under the action of $T$ on $\PP^5$, then we may write the action of $T$ on $M$ by analyzing how it acts on the normal bundle $N_{S/\PP^5}$.

\medskip\noindent
The surface $S$ is the image of the map $\nu:[x,y,z]\mapsto [x^2:xy:y^2:xz:yz:z^2]$. Three $T$-fixed points of $\PP^5$ lie on the Veronese surface $S=\nu(\PP^2)\subset \PP^5$, namely $p_0,p_2,p_5$. It suffices to analyze the action of $T$ on the fibers $\pi^{-1}(p_i)$, $i=0,2,5$, where $\pi:M\rightarrow \PP^5$ denotes the blowup map. Here, we only examine the case for the fiber $\pi^{-1}(p_0)$, as the other two cases are similar. The point $p_0$ is the image of $[1:0:0]$ under $\nu$, hence looking at the affine chart $p_0\ne 0$ of $\PP^5$, the derivative 
$$
d\nu_{|(0,0)}=\left( \begin{array}{ccccc}
1 & 0 & 0 & 0 & 0 \\
0 & 0 & 1 & 0 & 0
\end{array} \right)$$
tells us how the tangent space $T_{p_0}S$ sits inside the tangent space $T_{p_0}\PP^5$. Since we know that $\CC^*\! \times \CC^*$ acts on $T_{p_0}\PP^5$ by the diagonal matrix indicated in (\ref{TABLE1}), we conclude that $\CC^*\! \times \CC^*$ acts on both $T_{p_0}S$, and on the fiber over ${p_0}$ of the normal bundle $N_{S/\PP^5}$, by the following matrices, respectively
\begin{equation}\label{MATRICESI}
\left( \begin{array}{cc}
t_0^{-1}t_1 & \\
 & t_0^{-1} \end{array} \right), \quad \left( \begin{array}{ccc}
t_0^{-2}t_1^2 &  &  \\
  & t_1t_0^{-2} & \\
& &t_0^{-2}
\end{array} \right).\end{equation}
On the fiber $\pi^{-1}({p_0})$, we have three $T$-fixed points; let us denote them by $p_0',p_0'',p_0'''$. Similarly, we have six more points $p_2',p_2'',p_2'''$, and $p_5',p_5'',p_5'''$, from which we conclude that there are $12$ points on $M$ fixed by $T$; namely $p_0',p_0'',\ldots, p_5''',p_1,p_3,p_4$.

\medskip\noindent
The table \ref{TABLE1} lists the dimension of the BB-cells of $M$ with centers $p_1,p_3,p_4$. In order to complete the computation of the dimension of the BB-cells of $M$, we now compute the dimension of the BB-cells whose centers are $p_0',p_0'',\ldots ,p_5'''$.

\medskip\noindent
Let us compute the dimension of the BB-cell of $M$ whose center is $p_0'$. If $E$ stands for the exceptional divisor of the blowup $\pi:M\rightarrow \PP^5$, then we have that $T_{p_0'}M=T_{p_0'}E\oplus N_{E/M}(p_0')$, and this decomposition is compatible with the action of $T$. The space $T_{p_0'}E$ can be further decomposed as $T_{p_0'}E=T_{p_0'}S\oplus T_{p_0'}F$, where $T_{p_0'}S$ denotes tangent space to the Veronese surface $S=\nu(\PP^2)$, with the fiber $F=\pi^{-1}(p_0)$. This decomposition is also compatible with the action of $T$. Therefore, we have  
 $$T_{p_0'}M=T_{p_0'}E\oplus N_{E/M}(p_0')=T_{p_0'}S \oplus T_{p_0'}F\oplus N_{E/M}(p_0').$$

\medskip\noindent
In order to compute the dimension of the BB-cell at $p_0'$, we count the negative weights of the action of $T$ on each of the factor above. In fact, the matrices in (\ref{MATRICESI}) reveal that $T$ acts with one negative weight on $T_{p_0'}S$ and two negative weights on $T_{p_0'}F$. Since $N_{E/M}\cong \OO_E(-1)$ is the tautological bundle over $E$, we have that the action of $\CC^*\times \CC^*$ on the fiber $\OO_E(-1)_{p_0'}$ is given by the coordinate of $p_0'$, which is $t_0^{-2}t_1^2$, and has positive $T$-weight. We conclude that the dimension of the BB-cell whose center is $p_0'$ is $1+2+0=3$.

\medskip\noindent
The rightmost column of the following table lists the dimension of the BB-cells of $M$ induced by $T$. The centers of the BB-cells are listed in the leftmost column. We also list the negative weights of the $T$-action on each of the factors $T_{q}S , T_{q}F, N_{E/M}(q)$.
$$
\begin{tabular}{cccc|c}
\hline
point &$T_qS$ & $T_qF$ & $N_{E/M}(q)$ &cell-dimension\\ \hline
$p_0'$ & 1 & 2 & 0&3\\[2mm]
$p_0''$ & 1 & 1 & 0 &2\\[2mm]
$p_0'''$ & 1 & 0 & 1&2\\
  \hline
$p_2'$ & 2 & 2 & 1&5\\[2mm]
$p_2''$ & 2 & 1 & 1 &4\\[2mm]
$p_2'''$ & 2 & 0 & 1&3\\
\hline
$p_5'$ & 0 & 0 & 0&0\\[2mm]
$p_5''$ & 0 & 1 & 0 &1\\[2mm]
$p_5'''$ & 0 & 2 & 0&2\\
\hline
\end{tabular}
$$

\medskip\noindent
The BB-cells with centers at $p_1,p_3,p_4$ have dimensions $4,1,3$, respectively. Hence, there are three BB-cells of $M$ of dimensions $3$ and $2$. Also, there are two BB-cells of dimensions $1$ and $4$.

\medskip
\section{Cycles of $M$ of dimension $3$}

\medskip\noindent

\begin{theorem}\label{THM1} The space $A_3(M)\otimes \QQ$ has dimension $3$. The cone of effective $3$-cycles is generated by the following classes $$\mathrm{Eff}_3(M)=\langle H_1E_1, H_2E_2 , E_1E_2 \rangle .$$
\end{theorem}
\begin{proof}The dimension of $A_3(M)\otimes {\QQ}$ is $3$ since there are three BB-cells of dimension $3$ \cite{FULTON}. 

In order to describe the cone of effective $3$-cycles, we make use of the following strategy.
We know that $\langle H_1E_1, H_2E_2 , E_1E_2 \rangle \subset \mathrm{Eff}_3(M)$. We want to show that the opposite containment  holds, which is equivalent to showing the following inclusion of the dual cones $\langle H_1E_1, H_2E_2 , E_1E_2 \rangle^{\vee} \subset \mathrm{Eff}_3(M)^{\vee}$. By definition, we have $\overline{\mathrm{Eff}}_3(M)^{\vee}\!=\mathrm{Nef}^3(M)$. Hence, we get the desired inclusion of dual cones as long as the classes in $\langle H_1E_1, H_2E_2 , E_1E_2 \rangle^{\vee}$ are nef. Nefness is what we prove next.

By definition, the dual cone $\langle H_1E_1, H_2E_2 , E_1E_2 \rangle^{\vee}$ consists of all the classes $\alpha\in A_2(M)$ such that $\alpha\cdot \beta \ge 0$, where $\beta\in \langle H_1E_1, H_2E_2 , E_1E_2 \rangle$. 
In terms of a basis, this definition reads as follows. We need $a,b,c\in \QQ$ such that $(aH_1^2E_1+bH_2^2E_2+cH_1^3)(lH_1E_1+rH_2E_2 + sE_1E_2)\ge 0 \ ,$
where $r,l,s\ge 0$, and $\{H_1^2E_1, H_2^2E_2, H_1^3\}$ is a basis for $A_2(M)\otimes \QQ$ (the fact that this space has dimension 3 is proved in Theorem \ref{THM22}). 
We split up this inequality into three which, using Lemma \ref{CHAR} and $E_1=2H_1-H_2$ and $E_2=2H_2-H_1$, yields 
\begin{itemize}
\item[]
\begin{equation}
\begin{aligned}
H_1E_1.\alpha & = H_1E_1(aH_1^2E_1+bH_2^2E_2+cH_1^3) \nonumber\\
&=8b \ \geq 0,
\end{aligned}
\end{equation}

\item[]
\begin{equation}
\begin{aligned}
H_2E_2.\alpha &=H_2E_2(aH_1^2E_1+bH_2^2E_2+cH_1^3) \nonumber\\
 &=8a+6c \geq 0,
\end{aligned}
\end{equation}

\item[]
\begin{equation}
\begin{aligned}
E_1E_2.\alpha & = E_1E_2(aH_1^2E_1+bH_2^2E_2+cH_1^3) \nonumber\\
&=-8a-8b \geq 0 \ .
\end{aligned}
\end{equation}
\end{itemize}
From these three inequalities, it follows that the cone $\langle H_1E_1, H_2E_2 , E_1E_2 \rangle^{\vee}$ is generated by the following three classes,
$$\alpha_1=4H_1^3-3H_1^2E_1, \quad  \alpha_2=-3H_1^2E_1+3H_2^2E_2 +4H_1^3,\quad \alpha_3= H_1^3.$$
Since $A_2(M)\otimes \QQ$ has dimension 3, the linear span of $\{H_1^3,H_2^3,H_1^2H_2,H_2^2H_1\}$ must have at least one linear relation. Indeed, set $xH_1^3+yH_2^3+zH_1^2H_2+wH_2^2H_1=0$, for $x,y,z,w \in \QQ$. By intersecting this linear combination with the classes $H^2_1$, $H_1H_2$ and $E_1E_2$, we find linear relations among $x,y,z,w$. Such relations yield $2H_1^3-3H_1^2H_2+3H_1H_2^2-2H_2^3=0$, which implies that $2\alpha_1=(H_1+H_2)E_1E_2$ and $\alpha_2=4H_2^3$.

Since $H_1^3,H_2^3$ are nef classes, the result follows if the class $\tau=(H_1+H_2)E_1E_2$ is also nef. In other words, 
$\mbox{Eff}_3(M)=\langle H_1E_1, H_2E_2 , E_1E_2 \rangle$ if and only if the class $\tau$ is a nef class.
Nefness is proved in the next Lemma.
\end{proof}

\begin{lema}\label{tau}
The class $\tau=(H_1+H_2)E_1E_2\in A_2(M)$ is nef.
\end{lema}

\begin{proof} Throughout this proof $Z$ denotes an irreducible threefold in $M$. We wish to show that $Z.\tau\ge 0$ for any $Z$. 

The space $M$ is stratified into $SL_3\CC$-orbits and we can differentiate the irreducible threefolds in $M$ based on how they intersect these strata. We will prove this lemma by considering all the possibilities for these intersections and showing that in each of them, we have that $Z.\tau\ge 0$. The first case is the following. If $Z$ intersects transversally all the strata, then the intersection $Z\cap E_1E_2$ is an irreducible curve, which implies that $Z.\tau \ge0$, as $\tau$ is an ample divisor in $E_1E_2$.

Let us analyze the cases in which the intersection of $Z$ with the $SL_3$-strata may not be transversal.  Since $\tau\subset E_1E_2$, it suffices to analyze the dimensions of the intersections $Z\cap E_1$ and $Z\cap E_1E_2$ listed below. The rows of the following table show the possibilities we are going to examine.

$$
\begin{tabular}{|c|cc|}
\hline
case & $\mbox{dim }Z\cap E_1$  & $\mbox{dim }Z\cap E_1E_2$ \\ \hline
I & 2 &  2,1  \\[.1cm]\hline
II & 3 & 3,\  2\\
    \hline

  \end{tabular}
$$

\medskip
\begin{itemize}
\item[(I)] Suppose $\mbox{dim }Z\cap E_1 =2$, and $\mbox{dim }Z\cap E_1E_2 =1$. Since $H_1+H_2$ is the hyperplane section of the embedding $M\subset \PP^9\times \PP^{9*}$, then $\tau$ is a very ample divisor in $E_1E_2$. By Kleiman's Transversality Theorem \cite{KTT} the translate $Z^g\cap E_1E_2$, where $g\in U\subset SL_3$ and $U$ is open dense, is an irreducible curve. Then, $Z^g.\tau >0$ which implies that $Z.\tau\ge0$.

Suppose $\mbox{dim }Z\cap E_1=2$ and $\mbox{dim }Z\cap E_1E_2=2$. This implies the intersection $Z.\tau$ may be considered as taking place in $E_1$. In other words, there is a cycle $\tau'\in A_2(E_1)$ such that $\iota_*(\tau')=\tau$, where $\iota:E_1\rightarrow M$ stands for the inclusion map and $Z.\tau=Z_{|E_1}.\tau'$. Thus, it suffices to show that $\tau'$ is a nef class in $A_2(E_1)\otimes \QQ$. This is proved in Lemma \ref{HIL2}.

\item[(II)] Suppose $\mbox{dim }Z\cap E_1 =3$, and $\mbox{dim }Z\cap E_2 =3$. This, implies that $Z\subset E_1E_2$. Since $Z$ and $E_1E_2$ are both irreducible subvarieties, then $[Z]=\beta E_1E_2$ for some $\beta>0$. Hence, from Lemma \ref{CHAR}, it follows that $Z.\tau\ge 0$. 

Secondly, suppose $\mbox{ dim }Z\cap E_1 =3$ and $\mbox{dim }Z\cap E_2=2$. Since $Z$ is irreducible, then it is a divisor in $E_1$, and thus $[Z]$ can be written as a nonnegative linear combination of the generators of $\EFF_4(E_1)$. Namely, $[Z]=\alpha H_1E_1+\beta E_1E_2$, where $\beta, \alpha \ge 0$ by Remark \ref{HP2}. Consequently, using Lemma \ref{CHAR} and the relations $E_1=2H_1-H_2$ and $E_2=2H_2-H_1$, we get
\begin{equation*}
\begin{aligned}
Z.\tau &= (\alpha H_1E_1+\beta E_1E_2).\tau \\
&= (\alpha H_1E_1+\beta E_1E_2).(H_1+H_2)E_1E_2 \\
&= (\alpha H_1E_1+\beta E_1E_2).(H_1E_1E_2+H_2E_1E_2) \\
&= \alpha E_1E_2(H_1E_1.H_1+H_1E_1.H_2)+\beta(E_1E_2.H_1E_1E_2+E_1E_2.H_2E_1E_2) \\
&=(\beta E_1E_2)(H_1+H_2)E_1E_2 \ge 0.
\end{aligned}
\end{equation*}
\end{itemize}
\end{proof}

\medskip\noindent
The divisor $E_1$, which parametrizes double lines with two marked-points, is isomorphic to the Hilbert scheme of 2 points on the plane; denoted by $\mathbb{P}^{2 [2]}$. Indeed, the morphism $f:E_1\rightarrow \PP^{2[2]}$ that maps a double line with two marked-points to the subscheme of $\PP^2$ supported on such marked-points induces an isomorphism by Zariski Main Theorem.  Let $ch$ be the Hilbert-Chow morphism $$ch:\PP^{2[2]}\longrightarrow \PP^{2(2)},$$ where $\PP^{2(2)}$ is the symmetric product. Let us denote by $B=\mathrm{exc}(ch)$ the exceptional locus of the morphism $ch$. 

\begin{rmk}\label{HP2}
The cone of effective divisors of $\mathbb{P}^{2 [2]}$ is generated by the following divisor classes $\EFF^1(\mathbb{P}^{2 [2]})=\langle B, F\rangle$, where $F$ is the divisor which induces the morphism $g:E_1\rightarrow \PP^{2*}$ which maps a subscheme $Z\in \PP^{2[2]}$ to the unique line that contains it. 
Let $\iota:E_1\rightarrow M$ be inclusion map. Observe that the induced morphism on cycles $\iota_*:A_2(E_1)\otimes \QQ\rightarrow A_2(M)\otimes \QQ$ is injective (Lemma \ref{CODII}), and that $\iota_* B\sim E_1E_2$ as well as $\iota_* F\sim H_1E_1$. 
\end{rmk}

\medskip\noindent
We may interpret the morphisms $ch$ and $g$ in terms of the geometry of $M$. Let  $\pi_i:M\rightarrow \PP^5$, ($i=1,2$), be the projection maps from the definition of $M$. Observe that ${\pi_1}_{|E_1}:E_1\rightarrow \PP^{2*}$ is the morphism $g$, and ${\pi_2}_{|E_1}:E_1\rightarrow \Delta$ is the Hilbert-Chow morphism $ch$.

\medskip\noindent
Let us define two cycles of dimension $2$ in $E_1$. Let $\xi$ be the locus of $E_1$ $(\cong \PP^{2[2]})$ of non-reduced subschemes of length $2$ whose support lies on a fixed line. Let $\sigma$ be the locus of $E_1$ of non-reduced subschemes of length $2$ whose unique lines that contains it belongs to a fixed pencil of lines. These two cycles are further analyzed in Lemma \ref{CODII}. Recall that $\tau= (H_1+H_2)E_1E_2$ and that $\iota: E_1\rightarrow M$ denotes the inclusion map.

\medskip
\begin{lema}\label{HIL2} The class $[\tau']=[\sigma+\xi] \in A_2(E_1)\otimes \QQ$ is nef. Moreover, $\iota_*\tau'\sim \tau$.
\end{lema}
\begin{proof} 
Notice that the classes $[\sigma]\sim \iota^*H_1E_2$ and $[\xi]\sim \iota^*H_2E_2$ which implies that $\iota_*(\sigma+\xi)\sim \tau$. See Lemma \ref{CODII}.

In order to show nefness of the class $\tau'$, let us apply the same argument as in Lemma \ref{tau}. We are going to omit the details of the intersection theory of $E_1$, but the reader must keep in mind that the enumerative geometry of the elements of $E_1$ (for example, the appropriate analog of Lemma \ref{CHAR}), is necessary in order to carry out the following analysis. 

We want to show that $\tau' .Z\ge 0$ for every irreducible surface $Z\subset E_1$. Observe that the decomposition of $E_1$ into $SL_3$-orbits is the following $E_1=E_1^{\circ}\cup B$, where $B$ parametrizes double lines with a unique marked-point. Now, we analyze the possibilities for the surface $Z$ to intersect these strata. The first case is the following. Let us suppose that $Z\subset E_1$ intersects all the strata transversally. This implies that $Z\cap B$ is an irreducible curve, which implies that $Z.\tau'\ge 0$ as $\tau'$ is an ample divisor in $B$. 

Let us suppose that $\mathrm{dim } Z\cap B=2$. In this case, $Z\subset B$, which implies that $[Z]=a\sigma+b\xi$, for $a,b\ge 0$ rational numbers (Lemma \ref{CODII}). Therefore 
\begin{equation*}
\begin{aligned}
Z.\tau' &= (a \sigma + b \xi).\tau' \\
&= (a \sigma + b \xi).( \sigma +  \xi) = 0.
\end{aligned}
\end{equation*}

\end{proof}


\begin{cor}
The cone of numerically effective $2$-cycles is generated by the following classes $\mathrm{Nef}^3(M)=\langle H_1^3, H_2^3, \tau\rangle$,
where $\tau = (H_1+H_2)E_1E_2$.
\end{cor}

\medskip\noindent
The conclusion of this section is that the effective cone $\EFF_3(M)=\overline{\EFF}_3(M)$ is generated by three extremal cycles which parametrize the following families of complete conics: $H_2E_2$ parametrizes reducible conics with their node lying on a fixed line, its dual $H_1E_1$, which parametrizes double lines with two marked-points and such that these lines pass through a fixed point, and the cycle $E_1E_2$ which parametrizes double lines with a unique marked-point.

\section{Cycles of $M$ of dimension $2$}
\begin{theorem}\label{THM22}
The space $A_2(M)\otimes \QQ$ has dimension $3$. The cone of effective $2$-cycles $\mathrm{Eff}_2(M)$ is generated by the following classes
$$\mathrm{Eff}_2(M)=\langle H_1^2E_1,  H_2^2E_2, H_1E_1E_2, H_2E_1E_2\rangle.$$
\end{theorem}
\begin{proof} The space $A_2(M)\otimes \QQ$ has dimension $3$ since there are three BB-cells of dimension $2$ \cite{FULTON}.

It is clear that $\langle H_1^2E_1,  H_2^2E_2,  H_1E_1E_2, H_2E_1E_2\rangle \subset \mathrm{Eff}_2(M)$, and we wish to show that the opposite containment holds. In order to do that, we show the inclusion of the dual cones $\langle H_1^2E_1,  H_2^2E_2,  H_1E_1E_2, H_2E_1E_2\rangle^{\vee} \subset \mathrm{Eff}_2(M)^{\vee}$. By definition $\overline{\mathrm{Eff}}_2(M)^{\vee}=\mathrm{Nef}^2(M)$. Thus, we have the desired inclusion of dual cones as long as the classes in $\langle H_1^2E_1,  H_2^2E_2,  H_1E_1E_2,H_2E_1E_2\rangle^{\vee}$ are nef.
Nefness is what we prove next.

Note $\langle H_1^2E_1,  H_2^2E_2,  H_1E_1E_2, H_2E_1E_2\rangle^{\vee}=\langle H_1^2, H_1H_2, H_2^2, H_1^2-H_1H_2+H_2^2\rangle$. Indeed, in terms of the basis for $A_2(M)\otimes \QQ=\langle H_1^2E_1,H_2^2E_2,H_1^3\rangle$, we want to find $a,b,c\in \QQ$ such that
\begin{equation*}
\begin{aligned}
H_1^2.(aH_1^2E_1+bH_2^2E_2 +c H_1^3)&=4b+c \geq 0, \nonumber\\[2mm]
H_1H_2.(aH_1^2E_1+bH_2^2E_2+cH_1^3)&=2c \geq 0, \nonumber\\[2mm]
H_2^2.(aH_1^2E_1+bH_2^2E_2+cH_1^3)&=4a+4c \geq 0, \nonumber\\[2mm]
(H_1^2-H_1H_2+H_2^2).(aH_1^2E_1+bH_2^2E_2+cH_1^3)&= 4a+4b+3c\ge 0. \nonumber\\
\end{aligned}
\end{equation*}
Consequently, $\langle H_1^2, H_1H_2, H_2^2, H_1^2-H_1H_2+H_2^2\rangle^{\vee}=\langle H_1^2E_1,  H_2^2E_2,  H_1E_1E_2,  H_2E_1E_2\rangle$.

Since the classes $H_1^2, H_1H_2, H_2^2$ are nef, in order to finish the proof, it suffices to prove that $\beta=H_1^2-H_1H_2+H_2^2$ is also a nef class, which is Lemma \ref{BETA}.
\end{proof}

\medskip\noindent
\begin{lema}\label{BETA}
The class $\beta=H_1^2-H_1H_2+H_2^2$ is a nef class in $A_3(M)\otimes{\QQ}$.
\end{lema}
\begin{proof} Throughout this proof $Z$ denotes an irreducible surface in $M$. We wish to show that $Z.\beta\ge 0$ for any $Z$. 

The space $M$ is stratified into $SL_3\CC$-orbits and we can differentiate the irreducible surfaces in $M$ based on how they intersect these strata. We will prove this Lemma by considering all the possibilities for these intersections and showing that in each of them, we have that $Z.\beta\ge 0$. The first case is the following. Suppose $Z$ intersects transversally all the strata, thus the  intersection $Z\cap E_1$ and $Z\cap E_2$ is transversal. Moreover, suppose $\mbox{dim }Z\cap E_i=1$. This implies that $Z.2\beta \ge0$, as $Z.2\beta=Z.(H_1E_1+H_2E_2)=(Z.H_1)_{|E_1}+(Z.H_2)_{|E_2}$ and $(H_i)_{|E_i}$ are nef divisors in each $E_i$. 

The following table displays the cases to be analyzed. 
$$
\begin{tabular}{|c|cc|}
\hline
case & $\mbox{dim }Z\cap E_1$  &  $\mbox{dim }Z\cap E_2$ \\ 
\hline
I & 1 &  1  \\[.15cm]\hline
II & 2 & 2  \\ [.15cm]\hline
III & 2 & 1  \\ 
  \hline
\end{tabular}
$$
\begin{itemize}
\item[(I)]
In case $\mbox{dim }Z\cap E_1=1$ and $\mbox{dim }Z\cap E_2=1$, then there exists $g\in U\subset SL_3$ such that the translate $Z^g\cap E_i$ is an irreducible curve \cite{KTT}. At this point, we may apply the argument above and conclude that $Z^g.\beta\ge 0$, which implies that $Z.\beta\ge 0$.
\item[(II)]
If $\mbox{dim }Z\cap E_1=2$ and $\mbox{dim }Z\cap E_2=2$, then $Z$ is a divisor in $E_1E_2$. Hence, we can write $Z$ as the following nonnegative linear combination $Z=aH_1E_1E_2+bH_2E_1E_2$, for $a,b\ge 0$. Now, we compute 
\begin{equation}
\begin{aligned}
Z.\beta &= (aH_1E_1E_2+bH_2E_1E_2).(H_1^2-H_1H_2+H_2^2)=0.\nonumber\\
\end{aligned}
\end{equation}
\item[(III)] There is a symmetry in the table above, hence it suffices to consider the case, $\mbox{dim }Z\cap E_1=2$ and $\mbox{dim }Z\cap E_2=1$.

Since $\mbox{dim }Z\cap E_1=2$ and $Z$ is irreducible, then a $Z^g\subset E_1$, where $Z^g$ is a translate. Indeed, the boundary of $E_1^{\circ}$ is the unique closed $SL_3$-orbit $E_1E_2$. Hence, the closure of $Z^g\cap E$ lies in $E$.
Consequently, by Lemma \ref{CODII}, we can write $[Z^g]$ as the following nonnegative linear combination $Z^g=aH_1^2E_1+bH_1E_1E_2+cH_2E_1E_2$ for $a,b,c\ge 0$. We finish the proof by computing 

\begin{equation}
\begin{aligned}
Z^g.\beta &=(aH_1^2E_1+bH_1E_1E_2+cH_2E_1E_2).(H_1^2-H_1H_2+H_2^2)=0.\nonumber\\
\end{aligned}
\end{equation}
\end{itemize}
\end{proof}

\begin{lema}\label{CODII}	The space $A_2(E_1)\otimes \QQ$ has dimension $3$. The cone of effective $2$-cycles of $E_1$ is generated by the following classes $$\mathrm{Eff}_2(E_1)=\langle H_1^2E_1, H_1E_1E_2, H_2E_1E_2\rangle.$$
\end{lema}
\begin{proof} The divisor $E_1\subset M$ is $T$-invariant. Following the argument as in Section \ref{II}, we observe that the Bia\l{}ynicki-Birula decomposition of $E_1$, with respect to the restriction of the $T$-action on $E_1$, has three BB-cells of dimension $2$, hence the dimension of $A_2(E_1)\otimes \QQ$ is $3$.

The cycles classes $H_1^2E_1, H_1E_1E_2$ and $H_2E_1E_2$ are extremal effective cycles for the following reason. A generic point $c\in E_1$ parametrizes a double line $l$, with two marked-points. Thus, we have a morphism $\pi_1:E_1\rightarrow \PP^{2*}$ by sending $c\mapsto l$. The class $H_1^2E_1=(H_1^2)_{|E_1}$, in $A_2(E_1)$, is proportional to the class of the fiber $\pi^*(pt)$, which makes $H_1^2E_1$ an extremal effective cycle.
On the other hand, $E_1E_2$ is isomorphic to the flag variety $$\mathrm{Fl}(0,1)=\{(p,l)| p\in l\}\subset \PP^2\times \PP^{2*},$$
where the classes $H_1E_1E_2=(H_1)_{|E_1E_2}$ and $H_2E_1E_2=(H_2)_{|E_1E_2}$, in $A_2(E_1E_2)$, are proportional to the classes of the fibers of the projection morphisms into $\PP^2$ and $\PP^{2*}$. This implies that they are extremal effective cycles in $A_2(E_1E_2)$. Now, the inclusion morphism $\iota: E_1E_2\rightarrow M$ induces an injective map $\iota_*:A_2(E_1E_2)\rightarrow A_2(M)$. By definition of $M$, there are two projection maps $\pi_i:M\rightarrow \PP^5$, ($i=1,2$), and observe that the classes $H_1E_1E_2$ and $H_2E_1E_2$ are contracted by them. It follows that $H_1E_1E_2$ and $H_2E_1E_2$ are extremal cycles in $A_2(E_1)$ \cite{CC}. In order to show that these three cycles are all the extremal cycles, we compute the dual cone $\langle H_1^2E_1,H_1E_1E_2,H_2E_1E_2 \rangle^{\vee}$ and show it is generated by nef cycle classes. Observe that Lemma \ref{HIL2} carries this out partly (it shows that the wall generated by $H_1E_1E_2$ and $H_2E_1E_2$ is extremal), following the same strategy as in Lemma \ref{tau}. We omit the details as no difficulty arises. 
\end{proof}

\begin{cor}
The cone of numerically effective $3$-cycles is generated by the following classes $\mathrm{Nef}^2(M)=\langle H_1^2, H_2^2,H_1H_2, \beta\rangle$,
where $\beta = H_1^2-H_1H_2+H_2^2$.
\end{cor}

\medskip\noindent
The class $H_1^2E_1$ parametrizes fixed double lines with two marked-points which vary. Hence, $H_1^2E_1\cong \PP^{1[2]}$.
The class $H_2E_1E_2$ parametrizes double lines such that the unique marked-point in it lies in a fixed line. Hence, $H_2E_1E_2\cong \PP^{2*}$. 

\medskip\noindent
It is important to mention that Theorem \ref{THM1} and Theorem \ref{THM22} are consistent with the theory of spherical varieties. For example, both of the subvarieties in the previous paragraph are rational as the theory of spherical varieties predicted \cite{PERRIN}.

\medskip
\section*{Acknowledgements}
\medskip\noindent
I thank Izzet Coskun for suggesting the proof of the main result of this note. Thanks to Dawei Chen, Joe Harris and Francesco Cavazzani for useful conversations. I would also like to thank Chris Gomes for reading a draft of this note and generously improving the language of it. I am grateful to the department of mathematics at Harvard for providing me with ideal working conditions to complete this work. I thank the referee for a detailed reading of this note and useful suggestions.


\bigskip
\begin{thebibliography}{D}

\bibitem[1]{BB}A. ~Bia\l{}ynicki-Birula,
	\emph{Some theorems on actions of algebraic groups}, Ann. of
	Math.,
	\textbf{98}, (1973), pp.~480--497.
	


\bibitem[2]{CHAS}
M.~Chasles, \emph{Determination du nombre de sections coniques qui doivent
  toucher cinq courbes donn\'ees d'ordre quelconque, ou satisfaire \`a diverses
  autres conditions}, C.R. Acad. Sci. Paris \textbf{58} (1864), pp.~222--226.

\bibitem[3]{CC} D.~Chen and I. ~Coskun, \emph{Extremal higher codimension cycles on moduli spaces of curves}, Proc. London Math. Soc. \textbf{111} (2015), no. 1, pp.~ 181--204.


\bibitem[4]{DEL} O.~Debarre and L. ~Ein and R. Lazarsfeld and C. Voisin, \emph{Pseudoeffective and nef classes on abelian varieties}, Compos. Math., \textbf{147} No. 6, (2011), pp.~1793--1818. 


\bibitem[5]{DECON} C.~De~Concini and C.~Procesi, \emph{Complete symmetric varieties}, Invariant Theory (1980), Springer-Verlag.

\bibitem[6]{FULTON} W.~ Fulton, \emph{Intersection Theory},
Springer-{V}erlag, 1998.


\bibitem[7]{GH}
P.~Griffiths and J.~Harris, \emph{Principles of Algebraic Geometry}, Wiley
  {I}nterscience, 1978.

	

\bibitem[8]{JOE}
J.~Harris, \emph{Algebraic Geometry: a first course}, Springer-{V}erlag, Grad. Texts in Math., 1992.


\bibitem[9]{KTT}
S.~Kleiman, \emph{The transversality of a general translate}, Compos. {M}ath., \textbf{28} No. 3, (1974), pp.~287--297.


\bibitem[10]{CONICS}
S.~Kleiman and A.~Thorup, \emph{Intersection theory and enumerative geometry: a
  decade in review}, AMS Proc. of Symp. Pure Math \textbf{46-2} (1987), pp.~ 321--370.



\bibitem[11]{MUMII}
   D.~Mumford, \emph{Rational equivalence of $0$-cycles on surfaces}, J. Math. Kyoto Univ. 9-2 (1969) pp.~195--204.


\bibitem[12]{PERRIN}
  N.~Perrin,
\emph{On the geometry of spherical varieties},
Transform. Groups, \textbf{19}, (2014), pp.~ 171--223.
	

\bibitem[13]{SEV}
   F.~Severi, \emph{Sistemi d'equivalenza e corrispondenze algebriche sulle variet\`a algebriche}, Roma: Edizioni Cremonese, 1942-59.

\end{thebibliography}
\end{document}